\documentclass[12pt]{amsart}

\usepackage{mystyle}

\title{Weak Factorization of Hardy Spaces in the Bessel Setting}
\author{Roc Oliver}% \\ \small Universitat de Barcelona (UB) \\ \small roc.oliverv@gmail.com}
\address{Roc Oliver, Departament de Matem\`atiques i Inform\`atica, Universitat de Barcelona, Gran Via 585, 08007 Barcelona, Spain}
\email{roc.oliver@ub.edu}
\thanks{The first author would like thank WUSTL for the great hospitality he has received during his stay. R.~Oliver was supported by DGICYT Grant MTM2014-51834-P and DURSI Grant 2014SGR 289.}

\author{Brett D.~Wick}% \\ \small Universitat de Barcelona (UB) \\ \small roc.oliverv@gmail.com}
\address{Brett D. Wick, Department of Mathematics, Washington University--St.~Louis, One Brookings Drive, St.~Louis, MO 63130-4899, USA}
\email{wick@math.wustl.edu}
\thanks{B.~D.~Wick was partially supported by National Science Foundation DMS Grants \#1603246 and
\#1560955}

\begin{document}

\begin{abstract}
We provide the weak factorization of the Hardy spaces $H^{p}(\mathbb{R}_+, dm_{\lambda})$ in the Bessel setting, for $p\in \left(\frac{2\lambda + 1}{2\lambda + 2}, 1\right]$.  As a corollary we obtain a characterization of the boundedness of the commutator $[b, R_{\Delta_{\lambda}}]$ from $L^{q}(\mathbb{R}_+, dm_{\lambda})$ to $L^{r}(\mathbb{R}_+, dm_{\lambda})$ when $b\in \textrm{Lip}_{\alpha}(\mathbb{R}_+, dm_{\lambda})$ provided that $\alpha = \frac{1}{q} - \frac{1}{r}$. The results are a slight generalization and modification of the work of Duong, Li, Yang, and the second named author, which in turn are based on modifications and adaptations of work by Uchiyama.
\end{abstract}
\maketitle

\section{Introduction}

The theory of Hardy spaces has been studied and developed extensively in harmonic analysis and more precisely the theory of Hardy spaces on the Euclidean setting has been shown to have many applications, see \cite{Duren1970,Stein1993,FeffermanStein1972,CoifmanWeiss1977} for an instance of general references. 

%Recall that the classical Hardy space $H^{p}(\D)$, $0 < p <\infty$, on the unit disc $\D :=  \{z\in \C \colon \abs{z}<1\}$ is defined as the space of holomorphic functions $f = u + iv$, i.e., those satisfying the Cauchy-Riemann equations $\frac{\partial u}{\partial x} = \frac{\partial v}{\partial y}$ and $\frac{\partial u}{\partial y} = - \frac{\partial v}{\partial x}$ in $\D$, such that

%$$\norm{f}_{H^{p}(\D)} := \sup_{0 \leq r < 1} \pth{\frac{1}{2\pi} \int_{0}^{2\pi}\abs{f(r\ee^{it})}^{p} dt}^{\frac{1}{p}} < \infty.$$
%The so-called Riesz factorization theorem of Hardy space $H^{1}(\D)$, sometimes also called strong factorization, it is a well-known result. More concretely a function $f$ is in $H^{1}(\D)$ if and only if there exist $g,h \in H^{2}(\D)$ with $f = g\cdot h$ and $\norm{f}_{H^{1}(\D)} = \norm{g}_{H^{2}(\D)} \norm{h}_{H^{2}(\D)}$. This factorization result plays an important role in studying function theory and operator theory connected to the spaces $H^{1}(\D)$, $H^{2}(\D)$ and the space $BMO$, the space of bounded mean oscillation functions.

The real-variable Hardy space theory on $n$-dimensional Euclidean space $\R^{n}, n \geq 1$, plays an important role in harmonic analysis and has been systematically developed \cite{FeffermanStein1972,CoifmanWeiss1977}. There are many equivalent definitions of the Hardy spaces $H^{p}(\R^{n})$, $0 < p <\infty$.  It is well-known that when $p > 1$, the actual definition of $H^{p}(\R^{n})$ makes it equivalent to $L^{p}(\R^{n})$, but when $p \in (0,1]$, these spaces are much better suited to ask questions about harmonic analysis than are the $L^{p}(\R^{n})$ spaces, see \cite{Grafakos2014,Stein1993} for an account of all of this.

In the case of the real-variable Hardy space $H^{1}(\R^{n})$, Coifman, Rochberg and Weiss \cite{CoifmanRochbergWeiss1976} provided a factorization that works in studying function theory and operator theory of $H^{1}(\R^{n})$ which was called the weak factorization. This weak factorization for $H^{1}(\R^{n})$ consist of the following: every $f\in H^{1}(\R^{n})$ can be written as
$$f = \sum_{j=1}^{\infty} \sum_{i=1}^{n} (g_{j}^{i} R_{i} h_{j}^{i} + h_{j}^{i} R_{i}g_{j}^{i}),$$
where $\{g_{j}^{i}\}_{i,j}, \{h_{j}^{i}\}_{i,j} \in H^{2}(\R^{n})$ and $R_{i}$ are the Riesz transforms on $\mathbb{R}^n$ and $\left\Vert f\right\Vert_{H^1(\mathbb{R}^n)}\simeq \inf\left\{\sum_{j=1}^\infty\sum_{i=1}^{n} \left\Vert g_{j}^{i}\right\Vert_{L^2(\mathbb{R}^n)}\left\Vert g_{j}^{i}\right\Vert_{L^2(\mathbb{R}^n)}\right\}$, with the infimum taken over all possible representations of $f$ as above.  Later, Uchiyama \cite{Uchiyama1981} found an algorithmic way to generalize this weak factorization for Hardy spaces $H^{p}(\R^{n})$ with values of $p\in (0,1]$, but close to $1$.  In fact, the algorithm that Uchiyama provides works in spaces of homogeneous type for Calder\'on--Zygmund operators that satisfy certain lower bounds on their kernels.  On the other hand, it is also well-known, as pointed in \cite{CoifmanRochbergWeiss1976}, that this weak factorization is closely related with the boundedness of some commutator on $L^{p}$ spaces, to be defined later.  Since then, many authors generalized the boundedness of this commutator between different $L^{p}$ spaces \cite{Janson1978,PaluszynskiTaiblesonWeiss1991,Paluszynski1995,KarlovichLerner2005}.

The theory of the classical Hardy space $H^{p}(\R^{n})$ is intimately connected to the Laplacian $\Delta$.  Changing the differential operator introduces new challenge and directions to explore. In 1965, Muckenhoupt and Stein in \cite{MuckenhouptStein1965} introduced the notion of conjugacy associated with this Bessel operator $\Delta_{\lambda}$, $\lambda>0$, which is defined by
$$\Delta_{\lambda} f(x) := -\frac{d^{2}}{dx^{2}} f(x) - \frac{2\lambda}{x} \frac{d}{dx}f(x), \qquad x >0.$$
They developed a theory in the setting of $\Delta_{\lambda}$ which parallels the classical one associated to $\Delta$. Results on $L^{p}\dm$-boundedness of conjugate functions and fractional integrals associated with $\Delta_{\lambda}$ were obtained, where $p\in (1, \infty)$, $\R_{+}:= (0, \infty)$ and $d\m(x) := x^{2\lambda} dx$. Since then, many problems based on the Bessel context were studied; see, for example, \cite{Betancor2007,Betancor2010,Villani2008,YangYang2011,DuongLiWickYang2015}. In particular, the properties and $L^{p}$ boundedness $(1 < p <\infty)$ of Riesz transforms
$$R_{\Delta_{\lambda}} f := \partial_{x} (\Delta_{\lambda})^{-1/2}f$$
related to $\Delta_{\lambda}$ have been studied in \cite{MuckenhouptStein1965,Villani2008}. The related Hardy space
$$H^{1} \dm := \left\{f \in L^{1}\dm \colon R_{\Delta_{\lambda}} f \in L^{1}\dm\right\}$$
with norm $\norm{f}_{H^{1}\dm} := \norm{f}_{L^{1}\dm} + \norm{R_{\Delta_{\lambda}}f}_{L^{1}\dm}$ has been studied by Betancor et al. in \cite{BetancorDziubanskiTorrea2009} where they established the characterizations of the atomic Hardy space $H^{1}\dm$ associated with $\Delta_{\lambda}$ in terms of the Riesz transform and the radial maximal function associated with the Hankel convolution of a class $Z^{[\lambda]}$ of functions, which includes the Poisson semigroup and the heat semigroup as special cases. Duong, Li, Wick and Yang \cite{DuongLiWickYang2015} used Uchiyama's algorithm to prove the weak factorization on the Bessel setting of Hardy space $H^{1}\dm$ in terms of the Riesz transform $R_{\Delta_{\lambda}}$.  One can not appeal to Uchiyama's results directly since the kernel of the Bessel Riesz transforms do not satisfy the hypotheses of his results in \cite{Uchiyama1981}; however, with appropriate modifications one can carry out his approach.  

For the general case of $p\in (0,1]$, Yang and Yang \cite{YangYang2011} characterized the Hardy spaces $H^{p}\dm$ via atomic decomposition for values of $p\in \left(\frac{2\lambda + 1}{2\lambda + 2}, 1\right]$. They also showed a counterpart of the characterization of Hardy spaces in terms of the Riesz transforms. More concretely, in \cite[Theorem 1.2]{YangYang2011}, they proved that for $p\in \left(\frac{2\lambda + 1}{2\lambda + 2}, 1\right]$, we have that $f\in H^{p}\dm$ if and only if there exist $C>0$ such that
$$\norm{f \sharp_{\lambda} \phi_{\delta}}_{p} + \norm{R_{\Delta_{\lambda}} (f \sharp_{\lambda} \phi_{\delta})}_{p} \leq C,$$
where $\phi_{\delta} (x) = \a^{-2\lambda -1}\phi (x/\delta), \phi\in Z^{[\lambda]}$ and
$$f \sharp_{\lambda} g(x) := \int_{0}^{\infty} f(y) \tau_{x}^{[\lambda]} g(y) d\m(y),$$
where for $x\in (0,\infty)$, $\tau_{x}^{[\lambda]} g(y)$ denotes the Hankel translation of $g(y)$, that is,
$$\tau_{x}^{[\lambda]} g(y) := \frac{\Gamma (\lambda +\frac{1}{2})}{\Gamma(\lambda) \sqrt(\pi)} \int_{0}^{\pi} g(\sqrt{x^{2}+y^{2}-2xy\cos\theta})(\sin \theta)^{2\lambda-1} d\theta.$$

Then the aim of this paper is the following. We first build up a weak factorization for the Hardy space $H^{p}\dm$, for $p\in \left(\frac{2\lambda + 1}{2\lambda + 2}, 1\right]$, in terms of a bilinear form related to $R_{\Delta_{\lambda}}$. Secondly, as a consequence, we further prove that this weak factorization implies the characterization of the commutator with a symbol in $\lip\dm$, the space of Lipschitz (or H\"older) continuous of order $\a>0$ functions.

Throughout this paper, for any $x,r \in \Rp$, $I(x,r) := (x-r, x+r) \cap \Rp$. From the definition of the measure $\m$, i.e., $d\m(x) := x^{2\lambda} dx$, it is obvious that there exists a positive constant $C\in (1, \infty)$ such that for all $x, r \in \Rp$,
\begin{equation} \label{eq:measure}
	C^{-1} \m(I(x,r)) \leq x^{2\lambda}r + r^{2\lambda + 1} \leq C \m(I(x,r)).	
\end{equation}
Thus $(\Rp, \rho, d\m)$ is a space of homogeneous type in the sense of Coifman and Weiss \cite{CoifmanWeiss1977}, where $\rho(x,y):= \abs{x-y}$ for all $x,y\in \Rp$.  We denote by $\norm{\cdot}_{p}$ the norm of $L^{p}\dm$, for any $0 < p \leq \infty$.
 
 We now state our main result on the weak factorization of the Hardy spaces $H^{p}\dm$, for $p\in \left(\frac{2\lambda + 1}{2\lambda + 2}, 1\right]$.

\begin{theorem} \label{thetheorem}
	Let $p\in \left(\frac{2\lambda + 1}{2\lambda + 2}, 1\right]$ and $q, r >1$ such that
	\begin{equation}\label{eq:cond}
		\frac{1}{p} = \frac{1}{q} + \frac{1}{r}.
	\end{equation}
	For any $f\in H^{p}\dm$, there exists numbers $\{\alpha_{j}^{k}\}_{j,k}$, functions $\{g_{j}^{k}\}_{j,k} \subset L^{q}\dm$ and $\{h_{j}^{k}\}_{j,k} \subset L^{r}\dm$ such that
	\begin{equation}\label{eq:weakrep}
	f = \sum_{k=1}^{\infty} \sum_{j=1}^{\infty} \alpha_{j}^{k} \,\Pi (g_{j}^{k}, h_{j}^{k})
	\end{equation}
	in $H^{p}\dm$ where $\Pi$ is defined as
	$$\Pi(g,h) := g R_{\Delta_{\lambda}}h - h \tld{R_{\Delta_{\lambda}}} g,$$
	where $\tld{R_{\Delta_{\lambda}}}$ is the adjoint operator of $R_{\Delta_{\lambda}}$.
	Moreover, there exists a positive constant $C$ independent of $f$ such that
	\begin{align*}
		C^{-1} \norm{f}_{H^{p}\dm} \leq \inf\Bigg\{ & \pth{\sum_{k=1}^{\infty} \sum_{j=1}^{\infty} \abs{\a_{j}^{k}}^{p} \norm{g_{j}^{k}}_{q}^{p} \norm{h_{j}^{k}}_{r}^{p}}^{\frac{1}{p}} \colon \\
		& \left. f = \sum_{k=1}^{\infty} \sum_{j=1}^{\infty} \alpha_{j}^{k}\, \Pi (g_{j}^{k}, h_{j}^{k}) \right\} \leq C \norm{f}_{H^{p}\dm}.
	\end{align*}
\end{theorem}

Note that the case $p=1$ is exactly as Duong, Li, Wick and Yang \cite{DuongLiWickYang2015} did, so our contribution here is the cases of $p\in \left(\frac{2\lambda + 1}{2\lambda + 2}, 1\right)$.

As a corollary we get the following second main result that provides a characterization of the $\lip\dm$ space in terms of the boundedness of the commutators adapted to the Riesz transform $R_{\Delta_{\lambda}}$. Recall the definition of the $\lip\dm$ space associated with the Bessel operator, which is the dual space of $H^{p}\dm$, for $\a = \frac{1}{p} -1$, see \cite[Theorem B]{CoifmanWeiss1977}.
\begin{definition}[{\cite[p.591]{CoifmanWeiss1977}}]
Let $\a>0$. The space $\lip\dm$ is defined as the set of functions $f$ measurable on $\dm$ which satisfy
\begin{equation*}\label{eq:lip}
	\abs{f(x)-f(y)} \leq C m_\lambda(I)^{\a},
\end{equation*}
where $I$ is any interval containing both $x$ and $y$ and $C>0$ is a constant independent of $x$ and $y$. The greatest lower bound of these constants is denoted by $\norm{f}_{\lip\dm}$.
\end{definition}

Suppose $b\in L^{1}_{loc}\dm$ and $f\in L^{p}\dm$, for $p\in (1, \infty)$. Let $[b, R_{\Delta_{\lambda}}]$ be the commutator defined by
$$[b, R_{\Delta_{\lambda}}] f(x) := b(x) R_{\Delta_{\lambda}} f(x) - R_{\Delta_{\lambda}} (bf)(x).$$

\begin{theorem} \label{corollary}
	Let $b\in L_{loc}^{1}\dm$ and $1 < p < q <\infty$ and $0< \a < \frac{1}{2\lambda +1}$ such that
	 $$\a = \frac{1}{p} - \frac{1}{q}.$$
	\begin{enumerate}[(1)]
		\item If $b\in \lip\dm$, then the commutator $[b, R_{\Delta_{\lambda}}]$ is bounded from $L^{p}\dm$ to $L^{q}\dm$ with the operator norm \label{corollary1}
		$$\norm{[b, R_{\Delta_{\lambda}}]}_{L^{p}\dm\to L^{q}\dm} \leq C \norm{b}_{\lip\dm},$$
		where $C>0$ is a constant independent of $b$.
		\item If $[b, R_{\Delta_{\lambda}}]$ is bounded from $L^{p}\dm$ to $L^{q}\dm$, then $b \in \lip\dm$ and
		$$\norm{b}_{\lip\dm} \leq C \norm{[b, R_{\Delta_{\lambda}}]}_{L^{p}\dm\to L^{q}\dm},$$
		where $C>0$ is another constant independent of $b$. \label{corollary2}
	\end{enumerate}
\end{theorem}

The proofs of Theorems~\ref{thetheorem} and \ref{corollary} are shown via the following arguments. We first provide some preliminary results that we need for the proof of \thmref{thetheorem}, and the case \eqref{corollary1} of \thmref{corollary} is also proved before hand using similar ideas as \cite{DuongLiWickYang2015}. Once we have \thmref{thetheorem}, the second part \eqref{corollary2} of \thmref{corollary} follows from that.

The structure of the paper is as follows. In Section~\ref{preliminaries} we recall all the preliminary results from the literature that we will need later on. For example, we recall the Hardy spaces associated with $\Delta_{\lambda}$ and also we collect some fundamental estimates of the kernel of the Riesz transform $R_{\Delta_{\lambda}}$. In Section~\ref{proof} we prove Theorems \ref{thetheorem} and \ref{corollary}. In addition we also show some new lemmas that are the key to prove our main results. Finally, in Section~\ref{comments} we explain some difficulties we have encountered and open problems related to that.

Throughout the paper, if $f \lesssim g$ or $g \gtrsim f$ denote that there exists a constant $C$ independent of $f$ and $g$ such that $f \leq C g$. When $f \lesssim g$ and $g \lesssim f$ we write $f\simeq g$.

%
%% PRELIMINARIES
%

\section{Preliminaries} \label{preliminaries}

In this section we recall the preliminary notions that we need. More concretely, we recall the Hardy spaces and Riesz transform related to the Bessel operator $\Delta_{\lambda}$ from \cite{MuckenhouptStein1965,BetancorDziubanskiTorrea2009}.

We now recall the atomic characterization of the Hardy spaces $H^{p}\dm$, for $p\in \left(\frac{2\lambda + 1}{2\lambda + 2}, 1\right]$ in \cite{YangYang2011}.

\begin{definition}
	Let $p\in  \left(\frac{2\lambda + 1}{2\lambda + 2}, 1\right]$. A function $a$ is called an $H^{p}\dm$-atom (or simply a $p$-atom) if there exists an open bounded interval $I \subset \Rp$ such that $\supp a \subset I$, $\norm{a}_{\infty} \leq \m(I)^{-1/p}$ and $\int_{0}^{\infty} a(x) d\m(x) = 0$.
\end{definition}

Let $p\in  \left(\frac{2\lambda + 1}{2\lambda + 2}, 1\right]$. We point out that from \cite{YangYang2011}, the Hardy space $H^{p}\dm$ can be characterized via an atomic decomposition. That is, an $L^{p}\dm$ function $f\in H^{p}\dm$ if and only if
$$f = \sum_{k=1}^{\infty} \a_{k} a_{k} \quad \mathrm{in} \quad L^{p}\dm,$$
where for every $k$, $a_{k}$ is a $p$-atom and $\a_{k} \in \R$ satisfying that $\sum_{k=1}^{\infty} \abs{\a_{k}}^{p} < \infty$. Moreover, 
$$\norm{f}_{H^{p}\dm} \simeq \inf\left\{\pth{\sum_{k=1}^{\infty}\abs{\a_{k}}^{p}}^{\frac{1}{p}}\right\},$$
where the infimum is taken over all the decompositions of $f$ as above. 

We also note that $H^{p}\dm$ can also be characterized in terms of the radial maximal function, the nontangential maximal function, the grand maximal function, the Littlewood-Paley $g$-function and the Lusin-area function. More interesting in our case, we have also the Riesz transform characterization of these Hardy spaces; see \cite[Theorem 1.2]{YangYang2011}.

Next we recall the Poisson integral, the conjugate Poisson integral and the properties of the Riesz transforms. As in \cite{BetancorDziubanskiTorrea2009,DuongLiWickYang2015,YangYang2011}, let $\{P_{t}^{\lambda}\}_{t>0}$ be the Poisson semigroup $\{\ee^{-t \sqrt{\Delta_{\lambda}}}\}_{t>0}$ defined by
$$P_{t}^{[\lambda]} f(x) := \int_{0}^{\infty} P_{t}^{[\lambda]}(x,y) f(y) y^{2\lambda} dy,$$
where
$$P_{t}^{[\lambda]}(x,y) = \int_{0}^{\infty} \ee^{-tz}(xz)^{-\lambda+\frac{1}{2}} J_{\lambda-\frac{1}{2}}(xz)(yz)^{-\lambda+\frac{1}{2}}J_{\lambda-\frac{1}{2}}(yz) z^{2\lambda}dz$$
and $J_{\nu}$ is the Bessel function of the first kind and order $\nu$. Weinstein \cite{Weinstein1948} established the following formula for $P_{t}^{[\lambda]} (x,y)$: for $t, x, y \in \Rp$,
$$P_{t}^{[\lambda]}(x,y) = \frac{2\lambda t}{\pi} \int_{0}^{\pi} \frac{(\sin \theta)^{2\lambda-1}}{(x^{2}+y^{2}+t^{2}-2xy\cos \theta)^{\lambda+1}} d\theta.$$
The $\Delta_{\lambda}$-conjugate of the Poisson integral of $f$ is defined by
$$Q_{t}^{[\lambda]} f(x) := \int_{0}^{\infty} Q_{t}^{[\lambda]} (x,y) f(y) y^{2\lambda} dy,$$
where 
$$Q_{t}^{[\lambda]}(x,y) = -\frac{2\lambda}{\pi} \int_{0}^{\pi} \frac{(x-y\cos\theta)(\sin\theta)^{2\lambda-1}}{(x^{2}+y^{2}+t^{2}-2xy\cos \theta)^{\lambda+1}} d\theta.$$
From this, we deduce that for any $x,y\in \Rp$,
$$R_{\Delta_{\lambda}}(x,y) = \partial_{x} \int_{0}^{\infty} P_{t}^{[\lambda]}(x,y) dt = -\frac{2\lambda}{\pi} \int_{0}^{\pi} \frac{(x-y\cos\theta)(\sin\theta)^{2\lambda-1}}{(x^{2}+y^{2}+t^{2}-2xy\cos \theta)^{\lambda+1}} d\theta = \lim_{t\to 0} Q_{t}^{[\lambda]} (x,y).$$

We note that, as indicated in \cite{BetancorDziubanskiTorrea2009}, this Riesz transform $R_{\Delta_{\lambda}}$ is a standard Calder\'on-Zygmund operator. We recall that if $R_{\Delta_{\lambda}}
(x,y)$ is the kernel of the Riesz transform $R_{\Delta_{\lambda}}$ then for any measurable function $f$ in $\dm$,
$$R_{\Delta_{\lambda}} f(x) = \int_{\Rp} f(y) R_{\Delta_{\lambda}}(x,y) d\m(y), \qquad x\in \Rp.$$
\begin{proposition}[{\cite[Proposition 2.2]{DuongLiWickYang2015}}]\label{rieszcalderon}
The kernel $R_{\Delta_{\lambda}}(x,y)$ satisfies the following conditions:
\begin{enumerate}[(i)]
	\item For every $x, y\in \Rp$ with $x\not = y$, \label{rieszcalderon1}
	$$\abs{R_{\Delta_{\lambda}} (x,y)} \lesssim \frac{1}{\m(I(x, \abs{x-y}))};$$
	\item For every $x, x_{0}, y \in \Rp$ with $\abs{x_{0} - x} < \abs{x_{0}-y}/2$, \label{rieszcalderon2}
	$$\abs{R_{\Delta_{\lambda}}(y,x_{0}) - R_{\Delta_{\lambda}}(y,x)} \lesssim \frac{\abs{x_{0}-x}}{x_{0}-y} \frac{1}{\m(I(x_{0}, \abs{x_{0}-y}))}.$$
\end{enumerate}
\end{proposition}

Next we recall the following estimates of the kernel $R_{\Delta_{\lambda}}(x,y)$ of the Riesz transform $R_{\Delta_{\lambda}}$, which will be used later on.

\begin{proposition} \label{rieszestimate}
	The Riesz kernel $R_{\Delta_{\lambda}}(x,y)$ satisfies:
	\begin{enumerate}[(i)]
		\item {\rm \cite[Proposition 2.3]{DuongLiWickYang2015}} There exist $K_{1}\in (0,1)$ small enough and a positive constant $C_{K_{1}, \lambda}$ such that for any $x, y \in \Rp$ with $y < K_{1}x$, \label{rieszestimate1}
		$$R_{\Delta_{\lambda}}(x,y) \leq -C_{K_{1}, \lambda} \frac{1}{x^{2\lambda+1}}.$$
		\item {\rm \cite[Remark 2.4]{DuongLiWickYang2015}} There exist ${K}_{2}\in (0, 1/2)$ small enough and a positive constant $C_{K_{2}, \lambda}$ such that for any $x, y \in \Rp$ with $0 < y/x - 1 < {K}_{2}$, \label{rieszestimate2}
\begin{align*}
	R_{\Delta_{\lambda}}(x,y) & \geq \frac{1}{\pi} \frac{1}{x^{\lambda}y^{\lambda}}\frac{1}{x-y} - C_{K_{2}, \lambda} \frac{1}{x^{2\lambda+1}}\pth{1+\log_{+}\frac{\sqrt{xy}}{\abs{x-y}}} \\
	& \geq C_{{K}_{2},\lambda} \frac{1}{x^{\lambda}y^{\lambda}} \frac{1}{y-x},
\end{align*}
	where $\log_{+} (x) := \max\{0, \log(x)\}$.
	\end{enumerate}
\end{proposition}

%We will use the following well-known result that relationates the operator $\Pi$ with the commutator $[b, R_{\Delta_{\lambda}}]$, which is easy to prove but we include for completeness.
%
%\begin{lemma}\label{scalarproductpi}
%	Let functions $b\in \lip\dm$, $g \in L^{q}\dm$ and $h\in L^{r}\dm$. Then we have
%	$$\psc{b, \Pi(g,h)} = \psc{g, [b, R_{\Delta_{\lambda}}]h}.$$ 
%\end{lemma}
%
%\begin{proof}
%	By definition
%	$$\psc{b, \Pi(g,h)} = \psc{b, g R_{\Delta_{\lambda}}h } - \psc{b, h \tld{R_{\Delta_{\lambda}}} g} = \psc{g, b R_{\Delta_{\lambda}}h } - \psc{b, h \tld{R_{\Delta_{\lambda}}} g}.$$
%	Then, by Fubini's theorem,
%	\begin{align*}
%		\psc{b, h \tld{R_{\Delta_{\lambda}}} g} & = \int b(x) h(x) \tld{R_{\Delta_{\lambda}}} g(x) d\m(x) \\
%		& =  \int b(x) h(x) \int R_{\Delta_{\lambda}}(y,x) g(y) d\m(y) d\m(x) \\
%		& =  \int g(y) \int b(x) h(x) R_{\Delta_{\lambda}}(y,x) d\m(x) d\m(y) \\
%		& = \int g(y) R_{\Delta_{\lambda}} (b h)(y)d\m(y)  = \psc{g, R_{\Delta_{\lambda}} (bh)}.
%	\end{align*}
%	Therefore, 
%\end{proof}

\section{Proof of Main Result} \label{proof}

In this section we will provide the details of the main results of this article, the proof of Theorems~\ref{thetheorem} and~\ref{corollary}. Before that, we need some preliminary lemmas and results that we will use later on.

\begin{lemma} \label{lemma1}
	Let $f$ be a functions satisfying the following conditions:
	\begin{enumerate}[(i)]
		\item $\int_{0}^{\infty} f(x) x^{2\lambda} \d x = 0$;
		\item there exists $I(x_{1},r)$ and $I(x_{2},r)$ for some $x_{1}, x_{2} , r\in \R_{+}$ and $C_{1}, C_{2} >0$ such that
		$$\abs{f(x)} \leq C_{1} \chi_{I(x_{1},r)}(x) + C_{2} \chi_{I(x_{2},r)} (x);$$
		\item $\abs{x_{1}-x_{2}} \geq 4r.$
	\end{enumerate}
	Then there exists $C>0$ (independent of $x_{1}, x_{2}, r, C_{1}, C_{2}$) such that
	$$\norm{f}_{H^{p}\dm}^{p} \leq C \pth{\frac{\abs{x_{1}-x_{2}}}{r}}^{1-p} \log_{2} \left(\frac{\abs{x_{1}-x_{2}}}{r}\right) [ C_{1}^{p} m_{\lambda}(I(x_{1},r)) + C_{2}^{p} m_{\lambda}(I(x_{2},r)) ],$$
	for any $p\in  \left(\frac{2\lambda + 1}{2\lambda + 2}, 1\right]$.
\end{lemma}

\begin{proof}
	Assume that $f:= f_{1} + f_{2}$, where $\supp f_{i} \subset I(x_{i}, r)$ for $i=1,2$ (we can suppose that thanks to $(ii)$). We will show that $f$ has a particular $p$-atomic decomposition. To this end, we write
	$$f = \sum_{i=1}^{2} (f_{i} - \atld_{i}^{1} \chi_{I(x_{i}, 2r)}) + \sum_{i=1}^{2} \atld_{i}^{1} \chi_{I(x_{i}, 2r)} = f_{1}^{1} + f_{2}^{1}  +  \sum_{i=1}^{2} \atld_{i}^{1} \chi_{I(x_{i}, 2r)}$$
	where 
	$$\atld_{i}^{1} := \frac{1}{m_{\lambda} (I(x_{i},2r))} \int_{I(x_{i},r)} f_{i}(x) \d m_{\lambda}(x)$$
	and define $a_{i}^{1} := f_{i}^{1}/\a_{i}^{1}$, where $\a_{i}^{1} := \norm{f_{i}^{1}}_{\infty} \m(I(x_{i},2r))^{1/p}$, for $i=1,2$. Then we can see that $a_{i}^{1}$ is a $p$-atom supported on $I(x_{i}, 2r)$. Indeed, clearly $\supp a_{i}^{1} \subset I(x_{i},2r)$ and 
	$$\norm{a_{i}^{1}}_{\infty} = \frac{1}{\alpha_{i}^{1}} \norm{f_{i}^{1}}_{\infty} = \m(I(x_{i},2r))^{-1/p}.$$
	Finally, since 
	$$\int_{0}^{\infty} \atld_{i}^{1}\chi_{I(x_{i},2r)} (x) \d \m(x) = \int_{I(x_{i},r)}f_{i}(x) \d\m(x),$$
	we have that $\int_{0}^{\infty} a_{i}^{1}(x) \d\m(x) = 0$ and so $a_{i}^{1}$ is a $p$-atom. Moreover, since $\norm{f_{i}}_{\infty} \leq C_{i}$ (for $(ii)$) and $\m$ is doubling, we have that
	$$\abs{\a_{i}^{1}}  \lesssim 2^{1/p-1} C_{i} \m(I(x_{i},r))^{1/p}.$$
	
	For $i=1,2$, we further write
	$$\atld_{i}^{1}\chi_{I(x_{i},2r)} = \atld_{i}^{1} \chi_{I(x_{i},2r)} - \atld_{i}^{2} \chi_{I(x_{i},4r)} + \atld_{i}^{2} \chi_{I(x_{i},4r)} =: f_{i}^{2} + \atld_{i}^{2} \chi_{I(x_{i},4r)},$$
	where 
	$$\atld_{i}^{2} := \frac{1}{\m(I(x_{i},4r))} \int_{I(x_{i},r)} f_{i}(x) \d\m(x).$$
	Let $$\a_{i}^{2} := \norm{f_{i}^{2}}_{\infty} \m(I(x_{i}, 4r))^{1/p}$$
	and $a_{i}^{2} := f_{i}^{2}/ \a_{i}^{2}$. Then $a_{i}^{2}$ is a $p$-atom supported on $I(x_{i}, 4r)$. Indeed, it is obvious that $\supp a_{i}^{2} \subset I(x_{i},4r)$ and 
	$$\norm{a_{i}^{2}}_{\infty} = \frac{1}{\a_{i}^{2}} \norm{f_{i}^{2}}_{\infty} = \m(I(x_{i}, 4r))^{-1/p}.$$
	Finally, since
	$$\int_{0}^{\infty} \atld_{i}^{2}\chi_{I(x_{i},4r)} (x) \d \m(x) = \int_{I(x_{i},r)}f_{i}(x) \d\m(x),$$
	we have that $\int_{0}^{\infty} a_{i}^{2}(x) \d\m(x) = 0$ and so $a_{i}^{2}$ is a $p$-atom. Moreover, using \eqref{eq:measure} and the fact that $\m$ is doubling, we obtain that
	\begin{align*}
		\abs{\a_{i}^{2}} & \leq \abs{\atld_{i}^{1}} \m(I(x_{i},4r))^{1/p}\leq \frac{\m(I(x_{i}, r))}{\m(I(x_{i}, 2r))} \norm{f_{i}}_{\infty} \m(I(x_{i},4r))^{1/p} \\
		& \lesssim 2^{2(1/p-1)} C_{i}  \m(I(x_{i},r))^{1/p}.
	\end{align*}
	Continuing this fashion we see that
	$$f = \sum_{i=1}^{2} \left[ \sum_{j=1}^{J_{0}} f_{i}^{j} \right] + \sum_{i=1}^{2} \atld_{i}^{J_{0}} \chi_{I(x_{i}, 2^{J_{0}} r)} =  \sum_{i=1}^{2} \left[ \sum_{j=1}^{J_{0}} \a_{i}^{j} a_{i}^{j} \right] + \sum_{i=1}^{2} \atld_{i}^{J_{0}} \chi_{I(x_{i}, 2^{J_{0}} r)},$$
	where $J_{0}$ is the smallest integer larger than $\log_{2}\frac{\abs{x_{1}-x_{2}}}{r}$ and for $j \in \{2, 3, \ldots, J_{0}\}$,
	$$\atld_{i}^{j}:= \frac{1}{\m(I(x_{i}, 2^{j}r))} \int_{I(x_{i}, r)} f_{i}(x) \d\m(x),$$
	$$f_{i}^{j} := \atld_{i}^{j-1} \chi_{I(x_{i}, 2^{j-1}r)} - \atld_{i}^{j} \chi_{I(x_{i}, 2^{j}r)},$$
	$$\a_{i}^{j} := \norm{f_{i}^{j}}_{\infty} \m(I(x_{i}, 2^{j}r))^{1/p} \textrm{ and } a_{i}^{j}:= f_{i}^{j}/\a_{i}^{j}.$$
	By condition $(iii)$ we guarantee that $J_{0}\geq 2$.
	Moreover, for each $i$ and $j$, $a_{i}^{j}$ is a $p$-atom supported on $I(x_{i}, 2^{j}r)$. Indeed, for $j\in \{2,3, \ldots, J_{0}\}$, it is obvious that $\supp  a_{i}^{j} \subset I(x_{i},2^{j}r)$ and 
	$$\norm{a_{i}^{j}}_{\infty} = \frac{1}{\a_{i}^{j}} \norm{f_{i}^{j}}_{\infty} = \m(I(x_{i}, 2^{j}r))^{-1/p}.$$
	Finally, since
	$$\int_{0}^{\infty} \atld_{i}^{j}\chi_{I(x_{i},2^{j}r)} (x) \d \m(x) = \int_{I(x_{i},r)}f_{i}(x) \d\m(x),$$
	we have that $\int_{0}^{\infty} a_{i}^{j}(x) \d\m(x) = 0$ and so $a_{i}^{j}$ is a $p$-atom. In addition,  using again \eqref{eq:measure} and the fact that $\m$ is doubling we also have that
	\begin{align*}
		\abs{\a_{i}^{j}} & \leq \abs{\atld_{i}^{j-1}} \m(I(x_{i},2^{j}r))^{1/p}\leq \frac{\m(I(x_{i}, r))}{\m(I(x_{i}, 2^{j-1}r))} \norm{f_{i}}_{\infty} \m(I(x_{i},2^{j}r))^{1/p} \\
		& \lesssim 2^{j(1/p-1)} C_{i}  \m(I(x_{i},r))^{1/p}.
	\end{align*}
	For the last part, set
	\begin{align*}
		\sum_{i=1}^{2}& \atld_{i}^{J_{0}}  \chi_{I(x_{i}, 2^{J_{0}}r)} \\
		& = \left(\atld_{1}^{J_{0}} \chi_{I(x_{1}, 2^{J_{0}} r)} - \atld^{J_{0}} \chi_{I(\frac{x_{1}+x_{2}}{2}, 2^{J_{0}+1} r)}\right) + \left( \atld^{J_{0}} \chi_{I(\frac{x_{1}+x_{2}}{2}, 2^{J_{0}+1} r)} +  \atld_{2}^{J_{0}} \chi_{I(x_{2}, 2^{J_{0}} r)} \right) \\
		& := \sum_{i=1}^{2} f_{i}^{J_{0}+1},
	\end{align*}
	where
	\begin{align*}
		\atld^{J_{0}} & := \frac{1}{\m(I(\frac{x_{1}+x_{2}}{2}, 2^{J_{0}+1}r))} \int_{I(x_{1},r)} f_{1}(x) \d\m(x)\\ 
		& = -\frac{1}{\m(I(\frac{x_{1}+x_{2}}{2}, 2^{J_{0}+1}r))} \int_{I(x_{2},r)} f_{2}(x) \d\m(x),
	\end{align*}
	using $(i)$ in the last equality. Now, for $i=1,2$, let
	$$\a_{i}^{J_{0}+1} := \norm{f_{i}^{J_{0}+1}}_{\infty} \m(I(\frac{x_{1}+x_{2}}{2}, 2^{J_{0}+1}r))^{1/p}$$
	and
	$$a_{i}^{J_{0}+1} := f_{i}^{J_{0}+1}/\a_{i}^{J_{0}+1}.$$
	We see that $a_{i}^{J_{0}+1}$ is a $p$-atom: it is not difficult to see that $\supp  a_{i}^{J_{0}+1} \subset I(\frac{x_{1}+x_{2}}{2},2^{J_{0}+1}r)$ because $J_{0}$ is the smallest integer larger than $\log_{2} \frac{\abs{x_{1}-x_{2}}}{r}$ and 
	$$\norm{a_{i}^{J_{0}+1}}_{\infty} = \frac{1}{\a_{i}^{J_{0}+1}} \norm{f_{i}^{J_{0}+1}}_{\infty} = \m(I(\frac{x_{1}+x_{2}}{2}, 2^{J_{0}+1}r))^{-1/p}.$$
	Finally, since
	$$\int_{0}^{\infty} \atld^{J_{0}}\chi_{I(\frac{x_{1}+x_{2}}{2},2^{J_{0}+1}r)} (x) \d \m(x) = \int_{I(x_{1},r)}f_{1}(x) \d\m(x) = - \int_{I(x_{2},r)}f_{2}(x) \d\m(x),$$
	by $(i)$, we have that $\int_{0}^{\infty} a_{i}^{J_{0}+1}(x) \d\m(x) = 0$ and so $a_{i}^{J_{0}+1}$ is a $p$-atom. Moreover, by \eqref{eq:measure},
	\begin{align*}
		\abs{\a_{i}^{J_{0}+1}} & \leq \abs{\atld^{J_{0}}} \m(I(\frac{x_{1}+x_{2}}{2}, 2^{J_{0}+1}r))^{1/p} \\
		& \leq \frac{\m(I(\frac{x_{1}+x_{2}}{2}, 2^{J_{0}+1}r))^{1/p-1}}{\m(I(x_{i},r))^{1/p-1}} \norm{f_{i}}_{\infty} \m(I(x_{i},r))^{1/p} \\
		& \lesssim 2^{(J_{0}+1)(1/p-1)} C_{i} \m(I(x_{i},r))^{1/p}.
	\end{align*}
	
	In conclusion, we get the following $p$-atomic decomposition
	$$f = \sum_{i=1}^{2} \sum_{j=1}^{J_{0}+1} \a_{i}^{j} a_{i}^{j},$$
	with, for $i\in\{1, 2\}$,
	$$\abs{\a_{i}^{j}} \lesssim 2^{j(1/p-1)} C_{i} \m(I(x_{i},r))^{1/p}, \qquad j\in \{1, \ldots , J_{0}+1\}.$$
	By \cite{YangYang2011} we have that $f\in H^{p}\dm$ and
	\begin{align*}
		\norm{f}_{H^{p}\dm} & \leq \left( \sum_{i=1}^{2} \sum_{j=1}^{J_{0}+1} \abs{\a_{i}^{j}}^{p} \right)^{1/p} \lesssim \pth{\sum_{j=1}^{J_{0}+1}2^{j(1-p)} }^{1/p} \left( \sum_{i=1}^{2} C_{i}^{p} \m(I(x_{i},r)) \right)^{1/p}\\
		& \leq (J_{0}+1)^{1/p} (2^{(J_{0}+1)(1-p)})^{1/p} \left( \sum_{i=1}^{2} C_{i}^{p} \m(I(x_{i},r)) \right)^{1/p}\\
		&  \lesssim \left(\log_{2}\frac{\abs{x_{1}-x_{2}}}{r}\right)^{1/p} \pth{\frac{\abs{x_{1}-x_{2}}}{r}}^{1/p-1} \left( \sum_{i=1}^{2} C_{i}^{p} \m(I(x_{i},r)) \right)^{1/p}.
	\end{align*}
	This finishes the proof of Lemma~\ref{lemma1}.
\end{proof}

\begin{proposition} \label{proposition1}
	Let $p\in  \left(\frac{2\lambda + 1}{2\lambda + 2}, 1\right]$ and $q, r>0$ such that \eqref{eq:cond} holds. For every $\e > 0$, there exist $M>0$ and $C>0$ such that for all $p$-atom $a$, exists $g\in L^{q}\dm$ and $h \in L^{r}\dm$ satisfying that
	$$\norm{a - \Pi(g,h)}_{H^{p}\dm} < \e$$
	and $\norm{g}_{q} \norm{h}_{r} \leq CM^{\frac{2\lambda}{q}+1}$.
\end{proposition}

\begin{proof}
	Let $a$ be a $p$-atom with $\supp a \subset I(x_{0},r)$ where $x_{0}, r >0$. Observe that if $r>x_{0}$, then $I(x_{0}, r) = (x_{0}-r, x_{0}+r)\cap \R_{+} = I(\frac{x_{0}+r}{2}, \frac{x_{0}+r}{2})$. Therefore, without loss of generality, we may assume that 
	\begin{equation}\label{eq0}
		r \leq x_{0}.	
	\end{equation}
		
	Let $K_{1}$ and ${K}_{2}$ be the constants appeared in \eqref{rieszestimate1} and \eqref{rieszestimate2} of \propref{rieszestimate} respectively, and $K_{0} > \max\left\{\frac{1}{K_{1}}, \frac{1}{{K}_{2}}\right\} + 1>1$ large enough. For any $\e>0$, let $M$ be a positive constant large enough such that $M \geq 100 K_{0}$ and $\frac{\log_{2} M}{M^{2p-1}} < \e^{p}$ (possible since $p>1/2$ when $p\in \left(\frac{2\lambda + 1}{2\lambda + 2}, 1\right]$ ).
	
	We now consider the following two cases.
	
	{\bf Case (a):} Assume that $x_{0} \leq 2M r$. In this case, let $y_{0} := x_{0} + 2M K_{0}r$. Then, by \eqref{eq0},
	\begin{equation}\label{eq1}
		(1+K_{0}) x_{0} \leq y_{0} \leq (1+2MK_{0})x_{0}
	\end{equation}
	and
	\begin{equation}\label{eq2}
		2MK_{0}r \leq y_{0} \leq  (1+K_{0}) 2Mr.
	\end{equation}
	Define
	\begin{equation}\label{eq3}
		g(x):= \chi_{I(y_{0},r)} (x) \qquad \textrm{and}\qquad h(x):= -\frac{a(x)}{\tld{R_{\Delta_{\lambda}}} g(x_{0})}.
	\end{equation}
	By \eqref{eq0} and \eqref{eq1} we know that $y/x_{0} > K_{0} > K_{1}^{-1}$ for any $y\in I(y_{0}, r)$. Using this fact, \propref{rieszestimate} \eqref{rieszestimate1} and \eqref{eq2} we see that
	\begin{equation}\label{eq4}
		\abs{\tld{R_{\Delta_{\lambda}}} g(x_{0})} = \abs{\int_{y_{0}-r}^{y_{0}+r} R_{\Delta_{\lambda}}(y,x_{0})\d\m(y) } \gtrsim \int_{y_{0}-r}^{y_{0}+r} \frac{\d y}{y} \simeq \frac{r}{y_{0}} \simeq \frac{1}{M}.
	\end{equation}
	Moreover, from the definition of $g$ and $h$, and using \eqref{eq4} and \eqref{eq1}, it follows that
	\begin{align*}
		\norm{g}_{q} \norm{h}_{r} & \leq \frac{1}{\abs{\tld{R_{\Delta_{\lambda}}} g(x_{0})}} [\m (I(y_{0}, r))]^{1/q} [\m(I(x_{0}, r))]^{1/r - 1/p} \\
		& \lesssim M \pth{y_{0}^{2\lambda}r}^{1/q}\pth{x_{0}^{2\lambda}r}^{-1/q} \lesssim M^{\frac{2\lambda}{q}+1}.
	\end{align*}
	By the definition of the operator $\Pi$, we write
	$$a(x) - \Pi(g,h)(x) = a(x) \frac{\tld{R_{\Delta_{\lambda}}} g(x_{0}) - \tld{R_{\Delta_{\lambda}}} g(x)}{\tld{R_{\Delta_{\lambda}}} g(x_{0})} - g(x) R_{\Delta_{\lambda}} h(x) =: W_{1}(x) + W_{2}(x).$$
	Then it is obvious that $\supp W_{1} \subset I(x_{0},r)$ and $\supp W_{2} \subset I(y_{0}, r)$. From the cancellation property $\int_{0}^{\infty} a(y) \d \m(y) = 0$, the H\"older's regularity of the Riesz kernel $R_{\Delta_{\lambda}}(x,y)$ in \propref{rieszcalderon} \eqref{rieszcalderon2}, \eqref{eq4} and the fact that $\abs{x-y} \simeq \abs{x_{0}-y_{0}}$ for $y\in I(x_{0},r)$ and $x\in I(y_{0}, r)$, we have that
	\begin{align*}
		\abs{W_{2}(x)} & = \chi_{I(y_{0}, r)}(x) \abs{R_{\Delta_{\lambda}} h(x)} \\
		& \lesssim M  \chi_{I(y_{0}, r)}(x) \abs{ \int_{I(x_{0},r)} [R_{\Delta_{\lambda}} (x,y) - R_{\Delta_{\lambda}} (x,x_{0})] a(y) \d\m(y) } \\
		& \lesssim M  \chi_{I(y_{0}, r)}(x)  \int_{I(x_{0},r)} \frac{\abs{y-x_{0}}}{\abs{y-x}} \frac{\abs{a(y)}}{\m(I(y, \abs{y-x}))} \d\m(y) \\
		& \lesssim M  \chi_{I(y_{0}, r)}(x) \frac{r}{\abs{y_{0}-x_{0}}} \norm{a}_{\infty} \int_{I(x_{0},r)} \frac{\d\m(y)}{\m(I(y, \abs{y-x}))}\\
		& \lesssim C_{2} \chi_{I(y_{0}, r)}(x),
	\end{align*}
	where 
	$$C_{2} := \frac{\m(I(x_{0}, r))^{1-1/p}}{\m(I(y_{0}, \abs{y_{0}-x_{0}}))}.$$
	
	On the other hand, using again \eqref{eq4}, \propref{rieszcalderon} \eqref{rieszcalderon2} and the fact that $\abs{y-x_{0}} \simeq \abs{y_{0}-x_{0}}$ for $y\in I(y_{0}, r)$,
	\begin{align*}
		\abs{W_{1}(x)} & \lesssim M \chi_{I(x_{0}, r)}(x) \norm{a}_{{\infty}} \int_{I(y_{0},r)} \abs{R_{\Delta_{\lambda}}(y, x_{0}) - R_{\Delta_{\lambda}}(y,x) }\d\m(y) \\
		& \lesssim M \chi_{I(x_{0}, r)}(x) \frac{1}{\m(I(x_{0}, r))^{1/p}} \int_{I(y_{0},r)} \frac{\abs{x_{0}-x}}{\abs{x_{0}-y}} \frac{\d\m(y)}{\m(I(x_{0}, \abs{x_{0}-y}))} \\
		& \lesssim C_{1} \chi_{I(x_{0}, r)}(x),
	\end{align*}
	where
	$$C_{1} := \frac{\m(I(y_{0}, r)) \m(I(x_{0},r))^{-1/p}}{\m(I(x_{0}, \abs{x_{0}-y_{0}}))}.$$
	
	Moreover, using Fubini's theorem we can see that $\int_{0}^{\infty}\Pi(g,h)\d\m = 0$, then 
	$$\int_{0}^{\infty} [a - \Pi(g,h)]\d\m = 0.$$
	
	Hence, the function $f(x) := a(x) - \Pi(g,h)(x)$ satisfies all conditions in \lemmaref{lemma1}. Now from \lemmaref{lemma1}, we have that
	\begin{align*}
		\norm{a-\Pi(g,h)}_{H^{p}\dm}^{p} & \lesssim \pth{\frac{x_{0}-y_{0}}{r}}^{1-p} \log_{2}\pth{\frac{x_{0}-y_{0}}{r}} \times \\
		& \hspace{4cm} \times [C_{1}^{p} \m(I(x_{0},r)) + C_{2}^{p} \m(I(y_{0}, r))] \\
		& = \pth{\frac{x_{0}-y_{0}}{r}}^{1-p} \log_{2}\pth{\frac{x_{0}-y_{0}}{r}} \left[ \frac{\m(I(y_{0}, r))^{p}}{\m(I(x_{0}, \abs{x_{0}-y_{0}}))^{p}} \right. \\
		& \hspace{4cm} +\left. \frac{\m(I(x_{0}, r))^{p-1} \m(I(y_{0}, r))}{\m(I(y_{0}, \abs{y_{0}-x_{0}}))^{p}} \right]\\
		& \lesssim \pth{\frac{x_{0}-y_{0}}{r}}^{1-p} \log_{2}\pth{\frac{x_{0}-y_{0}}{r}} \frac{r^{p}}{\abs{x_{0}-y_{0}}^{p}} \\
		& \lesssim \frac{\log_{2}M}{M^{2p-1}} < \e^{p}.
	\end{align*}

	{\bf Case (b):} In this case we assume that $x_{0} > 2M r$ and let $y_{0}:= x_{0} - Mr/K_{0}$. Then it is clear that
	\begin{equation}\label{eq5}
		\frac{2K_{0}-1}{2K_{0}} x_{0} < y_{0} < x_{0}.
	\end{equation}
	Let $g$ and $h$ be as in Case (a) in \eqref{eq3}. For every $y\in I(y_{0}, r)$ and the fact that $K_{0}>\max\{1/K_{1}, 1/{K}_{2}\} + 1$ and $M\geq 100 K_{0}$ we have
	\begin{equation*}\label{eq6}
		0 < \frac{x_{0}}{y}-1 < {K}_{2}.
	\end{equation*}
	By \propref{rieszestimate} \eqref{rieszestimate2} and the fact that $y\simeq y_{0} \simeq x_{0}$ for $y\in I(y_{0},r)$, we conclude that
	\begin{equation}\label{eq7}
		\abs{\tld{R_{\Delta_{\lambda}}} g(x_{0})} \gtrsim \abs{\int_{y_{0}-r}^{y_{0}+r} \frac{1}{x_{0}^{\lambda}y_{0}^{\lambda}}\frac{\d\m(y)}{x_{0}-y} } \simeq \int_{y_{0}-r}^{y_{0}+r} \frac{\d y}{x_{0}-y_{0}} \simeq \frac{1}{M}. 
	\end{equation}
	Moreover, using the same operations as Case (a) and \eqref{eq5},
	$$\norm{g}_{q}\norm{h}_{r} \lesssim M \pth{y_{0}^{2\lambda}r}^{1/q}\pth{x_{0}^{2\lambda}r}^{-1/q}\lesssim M.$$
	Let $W_{1}, W_{2}, C_{1}$ and $C_{2}$ be the same as in Case (a). Then similarly, we obtain the same estimates $\abs{W_{1}(x)} \lesssim C_{1} \chi_{I(x_{0}, r)}(x)$ and $\abs{W_{2}(x)} \lesssim C_{2}\chi_{I(y_{0}, r)}(x)$ using \eqref{eq7} instead of \eqref{eq4}. Then by \lemmaref{lemma1} we obtain that
	\begin{align*}
		\norm{a-\Pi(g,h)}_{H^{p}\dm}^{p} & \lesssim \pth{\frac{x_{0}-y_{0}}{r}}^{1-p} \log_{2}\pth{\frac{x_{0}-y_{0}}{r}} \times \\
		& \hspace{4cm} \times [C_{1}^{p} \m(I(x_{0},r)) + C_{2}^{p} \m(I(y_{0}, r))] \\
		& = \pth{\frac{x_{0}-y_{0}}{r}}^{1-p} \log_{2}\pth{\frac{x_{0}-y_{0}}{r}} \left[ \frac{\m(I(y_{0}, r))^{p}}{\m(I(x_{0}, \abs{x_{0}-y_{0}}))^{p}} \right. \\
		& \hspace{4cm} +\left. \frac{\m(I(x_{0}, r))^{p-1} \m(I(y_{0}, r))}{\m(I(y_{0}, \abs{y_{0}-x_{0}}))^{p}} \right]\\
		& \lesssim \pth{\frac{x_{0}-y_{0}}{r}}^{1-p} \log_{2}\pth{\frac{x_{0}-y_{0}}{r}} \frac{r^{p}}{\abs{x_{0}-y_{0}}^{p}} \\
		& \lesssim \frac{\log_{2}M}{M^{2p-1}} < \e^{p}.
	\end{align*}
	which together with Case (a) completes the proof of \propref{proposition1}.
\end{proof}

We also need the following estimate of the bilinear operator $\Pi$.

\begin{proposition} \label{prop:pibounded}
	Let $p\in \left(\frac{1}{2}, 1\right)$ and $q, r>1$ such that \eqref{eq:cond} holds. There exists $C>0$ such that for any $g\in L^{q}\dm$ and $h\in L^{r}\dm$,
	$$\norm{\Pi(g,h)}_{H^{p}\dm} \leq C \norm{g}_{q} \norm{h}_{r}.$$
\end{proposition}

\begin{proof}
	%Since $R_{\Delta_{\lambda}}$ and $\tld{R_{\Delta_{\lambda}}}$ are both bounded on $L^{t}$ for any $t\in (1, \infty)$, we see that $\Pi(g,h)\in L^{p}$ for any $g\in L^{q}$ and $h\in L^{r}$?.
	Let $\a := \frac{1}{p} -1$. It is known \cite[Theorem B]{CoifmanWeiss1977} that the dual of $H^{p}\dm$ is $\lip \dm$. Then, by H\"older's inequality,
	\begin{align*}
		\norm{\Pi(g,h)}_{H^{p}\dm} & = \sup_{\substack{b \in \lip \dm,\\ \norm{b}_{\lip \dm} = 1}} \abs{\psc{\Pi(g,h), b}_{L^2(\mathbb{R}_+, dm_{\lambda})}} \\
		& = \sup_{\substack{b \in \lip \dm,\\ \norm{b}_{\lip \dm} = 1}} \abs{\psc{[b, R_{\Delta_{\lambda}}]h, g}_{L^2(\mathbb{R}_+, dm_{\lambda})}} \\
		& \leq \sup_{\substack{b \in \lip \dm,\\ \norm{b}_{\lip \dm} = 1}} \norm{[b, R_{\Delta_{\lambda}}]h}_{q'} \norm{g}_{q},
	\end{align*}
 	where $q'$ is the conjugate exponent of $q$. Now, since $b\in \lip \dm$ and using the usual kernel estimate in \propref{rieszcalderon} \eqref{rieszcalderon1} and \eqref{eq:measure} we obtain that
	\begin{align*}
		\abs{[b, R_{\Delta_{\lambda}}]h(x)} & \leq \int_{0}^{\infty}\abs{h(y)} \abs{R_{\Delta_{\lambda}} (x,y)} \abs{b(x)-b(y)} d\m(y) \\
		& \leq \norm{b}_{\lip \dm} \int_{0}^{\infty} \frac{\abs{h(y)} d\m(y)}{\m(I(x, \abs{x-y}))^{1-\a} }\\
		%& \lesssim  \norm{b}_{\lip \dm} x^{2\lambda (\a-1)} \int_{0}^{\infty} \frac{\abs{h(y)} d\m(y)}{\abs{x-y}^{1-\a} }\\
		& = \norm{b}_{\lip \dm} I_{\a}^{+} h(x),
	\end{align*}
	where $I_{\a}^{+}$ is the positive fractional integral operator defined by
	$$I_{\a}^{+} f(x) := \int_{0}^{\infty}\frac{\abs{f(y)} d\m(y)}{\m(I(x, \abs{x-y}))^{1-\a} }.$$
%	If $x<1$, this implies that
%	$$\abs{[b, R_{\Delta_{\lambda}}]h(x)} x^{2\lambda} \lesssim \norm{b}_{\lip \dm} I_{\a}^{+}h(x) x^{2\lambda \a} \leq \norm{b}_{\lip \dm} I_{\a}^{+}h(x) x^{2\lambda}.$$
%	Then
%	\begin{align*}
%		\int_{0}^{1} \abs{[b, R_{\Delta_{\lambda}}]h(x)}^{q'} d\m(x) & \leq \int_{0}^{1}\abs{[b, R_{\Delta_{\lambda}}]h(x)}^{q'} x^{2\lambda q'} dx \\
%		& \lesssim \norm{b}_{\lip \dm}^{q'} \int_{0}^{1} I_{\a}^{+}h(x)^{q'} x^{2\lambda (q'-1)}d\m(x) \\
%		& \leq  \norm{b}_{\lip \dm}^{q'} \int_{0}^{1} I_{\a}^{+}h(x)^{q'} d\m(x).
%	\end{align*}
%	On the other hand, if $x>1$, we can directly see that
%	\begin{align*}
%		\norm{[b, R_{\Delta_{\lambda}}]h\chi_{(1,\infty)}}_{q'}^{q'} & \lesssim \norm{b}_{\lip \dm}^{q'} \int_{1}^{\infty} I_{\a}^{+}h(x)^{q'} x^{2\lambda q'(\a-1)}d\m(x) \\
%		& \leq \norm{b}_{\lip \dm}^{q'} \norm{I_{\a}^{+}h\chi_{(1,\infty)}}_{q'}^{q'}.
%	\end{align*}
	It follows then that
	$$\norm{[b, R_{\Delta_{\lambda}}]h}_{q'}  \lesssim \norm{b}_{\lip \dm} \norm{I_{\a}^{+}h}_{q'}.$$ % prev  \cite[Theorem I.1 (i)]{GattoVagi1990}
	By \cite[Theorem 4.1]{BramantiCerutti1995}, $I_{\a}$ is bounded on $L^{r}\dm \to L^{q'}\dm$ provided that $1<r<1/\a$ with the condition $\frac{1}{q'} = \frac{1}{r} - \a$, which is true in our case by \eqref{eq:cond}. Note that in their proof, they actually prove that
	$$\norm{I_{\a}f}_{q'} \leq \norm{I_{\a}^{+}f }_{q'} \leq C \norm{f}_{r}, \qquad f\in L^{r}\dm.$$
	 Therefore,
	$$\sup_{\substack{b \in \lip \dm,\\ \norm{b}_{\lip \dm} = 1}} \norm{[b, R_{\Delta_{\lambda}}]h}_{q'} \lesssim \norm{h}_{r}$$ 
	and the proposition is proved. 
\end{proof}

Now we are in the situation to prove our main result.

\begin{proof}[Proof of \thmref{thetheorem}]
	By \propref{prop:pibounded} we have that for any $g\in L^{q}\dm$ and $h\in L^{r}\dm$,
	$$\norm{\Pi(g,h)}_{H^{p}\dm} \lesssim \norm{g}_{q} \norm{h}_{r}.$$
	From this, for any $f\in H^{p}\dm$ having the representation \eqref{eq:weakrep} with
	$$\pth{\sum_{k=1}^{\infty} \sum_{j=1}^{\infty} \pth{\abs{\a_{j}^{k}} \norm{g_{j}^{k}}_{q} \norm{h_{j}^{k}}_{r} }^{p} }^{1/p} < \infty$$
	it follows that
	\begin{align*}
		\norm{f}_{H^{p}\dm}^{p} & = \norm{\sum_{k=1}^{\infty} \sum_{j=1}^{\infty} \a_{j}^{k} \Pi(g_{j}^{k}h_{j}^{k})}_{H^{p}\dm}^{p} \\
		& \leq \sum_{k=1}^{\infty} \sum_{j=1}^{\infty} \abs{\a_{j}^{k}}^{p} \norm{\Pi(g_{j}^{k}, h_{j}^{k})}_{H^{p}\dm}^{p} \\
		& \lesssim \sum_{k=1}^{\infty} \sum_{j=1}^{\infty} \abs{\a_{j}^{k}}^{p} \norm{g_{j}^{k}}_{q}^{p} \norm{h_{j}^{k}}_{r}^{p}.
	\end{align*}
	Then, we have
	$$\norm{f}_{H^{p}\dm} \lesssim \inf \left\{  \pth{ \sum_{k=1}^{\infty} \sum_{j=1}^{\infty} \abs{\a_{j}^{k}}^{p} \norm{g_{j}^{k}}_{q}^{p} \norm{h_{j}^{k}}_{r}^{p}}^{\frac{1}{p}} \colon f = \sum_{k=1}^{\infty} \sum_{j=1}^{\infty} \alpha_{j}^{k} \Pi (g_{j}^{k}, h_{j}^{k}) \right\}.$$
	
	To see the converse, let $f\in H^{p}\dm$. We will show that $f$ has a representation as in \eqref{eq:weakrep} with
	 \begin{equation}\label{eq:norm}
	 	\inf \left\{ \pth{ \sum_{k, j=1}^{\infty} \pth{\abs{\a_{j}^{k}} \norm{g_{j}^{k}}_{q} \norm{h_{j}^{k}}_{r}}^{p}}^{\frac{1}{p}} \colon f = \sum_{k, j=1}^{\infty} \alpha_{j}^{k} \Pi (g_{j}^{k}, h_{j}^{k}) \right\} \lesssim \norm{f}_{H^{p}\dm}.
	\end{equation}
	 To this end, assume that $f$ has the following atomic representation $f = \sum_{j=1}^{\infty} \a_{j}^{1}a_{j}^{1}$ with $\pth{\sum_{j} \abs{\a_{j}^{1}}^{p}}^{1/p} \leq C \norm{f}_{H^{p}\dm}$ for certain constant $C \in (1, \infty)$ (see \cite[Theorem 1.1]{YangYang2011}), where $\{\a_{j}^{1}\}_{j}$ are numbers and $\{a_{j}^{1}\}_{j}$ are $p$-atoms.  
	 
	 First of all, for given $\e>0$, such that $\e C < 1$, and $a_{j}^{1}$, by \propref{proposition1}, there exist $g_{j}^{1}  \in L^{q}\dm$ and $h_{j}^{1} \in L^{r}\dm$ with
	 $$ \norm{g_{j}^{1}}_{q} \norm{h_{j}^{1}}_{r} \lesssim M^{\frac{2\lambda}{q}+1}$$
	 and
	 $$\norm{a_{j}^{1} - \Pi(g_{j}^{1}, h_{j}^{1})}_{H^{p}\dm} < \e.$$
	 Now we write 
	 $$f = \sum_{j=1}^{\infty} \a_{j}^{1}a_{j}^{1} = \sum_{j=1}^{\infty} \a_{j}^{1} \Pi(g_{j}^{1}, h_{j}^{1}) + \sum_{j=1}^{\infty} \a_{j}^{1}\left[a_{j}^{1} - \Pi(g_{j}^{1}, h_{j}^{1})\right] =: M_{1} + E_{1}.$$
	 Observe that 
	 $$\norm{E_{1}}_{H^{p}\dm}^{p} \leq \sum_{j=1}^{\infty} \abs{\a_{j}^{1}}^{p} \norm{a_{j}^{1} - \Pi(g_{j}^{1}, h_{j}^{1})}_{H^{p}\dm}^{p} \leq \e^{p} C^{p}\norm{f}_{H^{p}\dm}^{p}.$$
	 
	 Since $E_{1}\in H^{p}\dm$, by \cite[Theorem 1.1]{YangYang2011} again, there exist a sequence of $p$-atoms $\{a_{j}^{2}\}_{j}$ and numbers $\{\a_{j}^{2}\}_{j}$ such that $E_{1} = \sum_{j=1}^{\infty} \a_{j}^{2} a_{j}^{2}$ and
	 $$\pth{\sum_{j=1}^{\infty} \abs{\a_{j}^{2}}^{p}}^{1/p} \leq C \norm{E_{1}}_{H^{p}\dm} \leq \e C^{2} \norm{f}_{H^{p}\dm}.$$
	 
	 Another application of \propref{proposition1} with $\e$ and $a_{j}^{2}$ implies that there exist functions $g_{j}^{2} \in L^{q}\dm$ and $h_{j}^{2} \in L^{r}\dm$ with
	 $$ \norm{g_{j}^{2}}_{q} \norm{h_{j}^{2}}_{r} \lesssim M^{\frac{2\lambda}{q}+1} $$ 
	 and 
	 $$\norm{a_{j}^{2} - \Pi(g_{j}^{2}, h_{j}^{2})}_{H^{p}\dm} < \e.$$
	 Thus, we have
	 $$E_{1} = \sum_{j=1}^{\infty} \a_{j}^{2} a_{j}^{2} = \sum_{j=1}^{\infty} \a_{j}^{2} \Pi(g_{j}^{2}, h_{j}^{2}) + \sum_{j=1}^{\infty} \a_{j}^{2}[a_{j}^{2} - \Pi(g_{j}^{2}, h_{j}^{2})] =: M_{2} + E_{2}.$$
	 Moreover, 
	 $$\norm{E_{2}}_{H^{p}\dm}^{p} \leq \sum_{j=1}^{\infty} \abs{\a_{j}^{2}}^{p} \norm{a_{j}^{2} - \Pi(g_{j}^{2}, h_{j}^{2})}_{H^{p}\dm}^{p} \leq \e^{2p} C^{2p}\norm{f}_{H^{p}\dm}^{p}.$$
	 Then, we conclude that
	 $$f = \sum_{j=1}^{\infty} \a_{j}^{1} a_{j}^{1} = \sum_{k=1}^{2} \sum_{j=1}^{\infty} \a_{j}^{k} \Pi(g_{j}^{k}, h_{j}^{k}) + E_{2}.$$
	 
	 Continuing in this way, we deduce that for any $K\in \N$, $f$ has the following representation
	 $$f = \sum_{k=1}^{K} \sum_{j=1}^{\infty} \a_{j}^{k} \Pi(g_{j}^{k}, h_{j}^{k}) + E_{K}$$
	 satisfying, for any $k\in \{1, \ldots, K\}$,
	 $$ \norm{g_{j}^{k}}_{q} \norm{h_{j}^{k}}_{r} \lesssim M^{\frac{2\lambda}{q}+1},$$
	 $$\pth{\sum_{j=1}^{\infty} \abs{\a_{j}^{k}}^{p}}^{1/p} \leq \e^{k-1} C^{k} \norm{f}_{H^{p}\dm}$$
	 and
	 $$\norm{E_{K}}_{H^{p}\dm} \leq (\e C)^{K}\norm{f}_{H^{p}\dm}.$$
	 Thus, letting $K \to \infty$, we see that \eqref{eq:weakrep} holds. Moreover, since $\e C <1$, we have that
	 $$\sum_{k=1}^{\infty} \sum_{j=1}^{\infty} \abs{\a_{j}^{k}}^{p} \leq \sum_{k=1}^{\infty} \e^{-p} (\e C)^{kp} \norm{f}_{H^{p}\dm}^{p} \lesssim \norm{f}_{H^{p}\dm}^{p}$$
	 which implies \eqref{eq:norm} and hence, completes the proof of \thmref{thetheorem}.
\end{proof}
	
Finally, we can prove \thmref{corollary} that follows from \thmref{thetheorem}. 

\begin{proof}[Proof of \thmref{corollary}]
\begin{enumerate}
	\item This is already proved in the proof of Proposition~\ref{prop:pibounded}.
	\item  Assume that $[b, R_{\Delta_{\lambda}}]$ is bounded from $L^{p}\dm$ to $L^{q}\dm$ with norm $\norm{[b, R_{\Delta_{\lambda}}]}_{p\to q} := \norm{[b, R_{\Delta_{\lambda}}]}_{L^{p}\dm \to L^{q}\dm}$. Let $f\in H^{t}\dm$ such that $$\frac{1}{t} = \frac{1}{p} + \frac{1}{q'},$$
	where $q'$ is the conjugate exponent of $q$. By hypothesis and \thmref{thetheorem}, there exists numbers $\{\a_{j}^{k}\}_{j,k}$, functions $\{g_{j}^{k}\}_{j,k} \subset L^{q'}\dm$ and $\{h_{j}^{k}\}_{j,k} \subset L^{p}\dm$ such that
	
	\begin{eqnarray*}
	\langle b, f \rangle_{L^2(\mathbb{R}_+, dm_{\lambda})}  & = & \sum_{k=1}^{\infty} \sum_{j=1}^{\infty}  \a_{j}^{k} \left\langle b, \Pi \pth{g_{j}^{k}, h_{j}^{k}} \right\rangle_{L^2(\mathbb{R}_+, dm_{\lambda})}\\
	& = & \sum_{k=1}^{\infty} \sum_{j=1}^{\infty} \a_{j}^{k} \left\langle g_{j}^{k}, [b, R_{\Delta_{\lambda}}] h_{j}^{k}\right\rangle_{L^2(\mathbb{R}_+, dm_{\lambda})}.
	\end{eqnarray*}

	By H\"older's inequality, the hypothesis, the fact that $t<1$ and \thmref{thetheorem} again, it implies that
	\begin{align*}
		\abs{\langle b, f \rangle_{{L^2(\mathbb{R}_+, dm_{\lambda})}}} & \leq \sum_{k=1}^{\infty} \sum_{j=1}^{\infty} \abs{\a_{j}^{k}} \norm{g_{j}^{k}}_{q'} \norm{[b, R_{\Delta_{\lambda}}] h_{j}^{k}}_{q} \\
		& \leq \norm{[b, R_{\Delta_{\lambda}}]}_{p\to q} \sum_{k=1}^{\infty} \sum_{j=1}^{\infty} \abs{\a_{j}^{k}} \norm{g_{j}^{k}}_{q'}\norm{h_{j}^{k}}_{p} \\
		& \leq \norm{[b, R_{\Delta_{\lambda}}]}_{p\to q} \pth{ \sum_{k=1}^{\infty} \sum_{j=1}^{\infty} \pth{ \abs{\a_{j}^{k}} \norm{g_{j}^{k}}_{q'}\norm{h_{j}^{k}}_{p}}^{t}}^{\frac{1}{t}} \\
		& \lesssim \norm{[b, R_{\Delta_{\lambda}}]}_{p\to q}  \norm{f}_{H^{t}\dm}.
	\end{align*}
	Then, if $\a = \frac{1}{t}-1 = \frac{1}{p}-\frac{1}{q}$, which is true by hypothesis, by duality \cite[Theorem B]{CoifmanWeiss1977} between $H^{t}\dm$ and $\lip\dm$, we finish the proof of \thmref{corollary}. \qedhere
\end{enumerate}
\end{proof}

\section{Comments} \label{comments}

We want to point out that the proofs here only work for $p$ values that are large enough. Namely, we have proved the weak factorization of Hardy spaces $H^{p}\dm$ for values of $p\in \left(\frac{2\lambda + 1}{2\lambda + 2}, 1\right]$. A reasonable question would be what happens when $p$ is small. More concretely for values of $p\in \left(0, \frac{1}{2}\right]$. This is an open problem and we do not know exactly how to proceed to solve this question. In fact, in order to solve that we need a crucial result that characterizes the Hardy spaces $H^{p}\dm$ and the atomic ones for values of $p\in \left(0,\frac{1}{2}\right)$ and we think this is not trivial at all. We need more conditions on the atoms, that is, we need to impose more vanishing moment conditions as \cite{CoifmanWeiss1977} suggested.

\bibliographystyle{acm}
\bibliography{wfhs}
\addcontentsline{toc}{chapter}{Bibliography}

\end{document}